\documentclass[reqno,12pt]{amsart}

\usepackage[utf8]{inputenc}

\usepackage{enumerate}
\usepackage[margin=1in]{geometry}
\usepackage{ifpdf}
\usepackage{amsmath}
\usepackage{amsfonts}
\usepackage{amssymb}
\usepackage{amsthm}
\usepackage[ocgcolorlinks,hyperfootnotes=false,colorlinks=true,citecolor=blue,linkcolor=blue,urlcolor=blue]{hyperref}
\usepackage{setspace}
\usepackage{amsrefs}
\usepackage{nicefrac}
\usepackage{graphicx}
\usepackage{color}
\usepackage{mathtools}

\usepackage[T1]{fontenc}

\newcommand{\ignore}[1]{}

\renewcommand{\Re}{\operatorname{Re}}
\renewcommand{\Im}{\operatorname{Im}}

\newcommand{\C}{{\mathbb{C}}}
\newcommand{\R}{{\mathbb{R}}}

\newcommand{\sA}{{\mathcal{A}}}

\newcommand{\sF}{{\mathcal{F}}}

\newcommand{\sS}{{\mathcal{S}}}
\newcommand{\sI}{{\mathcal{I}}}

\newcommand{\sV}{{\mathcal{V}}}
\newcommand{\sX}{{\mathcal{X}}}
\newcommand{\sY}{{\mathcal{Y}}}

\newtheorem{thm}{Theorem}[section]

\newtheorem{prop}[thm]{Proposition}

\newtheorem{lemma}[thm]{Lemma}

\theoremstyle{definition}
\newtheorem{defn}[thm]{Definition}
\newtheorem{example}[thm]{Example}

\theoremstyle{remark}
\newtheorem{remark}[thm]{Remark}

\author{Ji\v{r}\'{\i} Lebl}
\thanks{The author was in part supported by Simons Foundation collaboration grant 710294.}
\address{Department of Mathematics, Oklahoma State University,
Stillwater, OK 74078, USA}
\email{lebl@okstate.edu}

\date{August 13, 2021}

\title{Segre-Degenerate Points Form a Semianalytic Set}

\keywords{Segre-degenerate, Segre variety, semianalytic}
\subjclass[2020]{32C07 (Primary), 32B20, 14P15 (Secondary)}

\begin{document}

%\doublespace

\begin{abstract}
We prove that the set of Segre-degenerate points of a real-analytic subvariety $X$ in $\C^n$ is a closed semianalytic set.  It is a subvariety if $X$ is coherent.  More precisely, the set of points where the germ of the Segre variety is of dimension $k$ or greater is a closed semianalytic set in general, and for a coherent $X$, it is a real-analytic subvariety of $X$.  For a hypersurface $X$ in $\C^n$, the set of Segre-degenerate points, $X_{[n]}$, is a semianalytic set of dimension at most $2n-4$.  If $X$ is coherent, then $X_{[n]}$ is a complex subvariety of (complex) dimension $n-2$.  Example hypersurfaces are given showing that $X_{[n]}$ need not be a subvariety and that it also need not be complex; $X_{[n]}$ can, for instance, be a real line.
\end{abstract}

\maketitle

%\enlargethispage{\baselineskip}

%%%%%%%%%%%%%%%%%%%%%%%%%%%%%%%%%%%%%%%%%%%%%%%%%%%%%%%%%%%%%%%%%%%%%%%%%%%%

\section{Introduction} \label{section:intro}

Segre varieties are a widely used tool for dealing with real-analytic
submanifolds in complex manifolds.
Recently, there have been many applications of Segre variety techniques
to singular real-analytic subvarieties,
and while the techniques are powerful, they have to be applied carefully.
It is tempting to cite an argument or result for submanifolds to prove
the same result for subvarieties, but there are two things that can go wrong.
First, the Segre variety can be degenerate (of wrong dimension),
and second, the variety itself may be not coherent, and the Segre variety cannot
be defined by the same function(s) at all points.
One cannot define Segre varieties with respect to
the complexification at one point and expect this complexification to give a well-defined Segre
variety at all nearby points (germs have complexifications, but
their representatives may not).
One incorrect but very tempting statement
is that the set of Segre-degenerate points of a real hypersurface in $\C^n$
is necessarily a complex-analytic subvariety.
The result follows for coherent 
hypersurfaces, but not in general.
The set of Segre-degenerate
points of a hypersurface is not only not a
complex-analytic subvariety in general,
it need not even be a real-analytic subvariety, it is
merely a semianalytic set.
We give an example where it is not a subvariety, and one where it
is of odd real dimension.

The idea of using Segre varieties is old, although the techniques for using
them in CR geometry were brought into prominence first by
Webster~\cite{Webster:map} and Diederich--Forn\ae ss~\cite{DF:realbnd}.
For a good introduction to their use for submanifolds, see the book by
Baouendi--Ebenfelt--Rothschild~\cite{BER:book}.
They started to be used for singular subvarieties recently,
see for example
Burns--Gong~\cite{burnsgong:flat},
Diederich--Mazzilli~\cite{DM},
the author~\cite{Lebl:lfsing},
Adamus--Randriambololona--Shafikov~\cite{ARS},
Fern\'{a}ndez-P\'{e}rez~\cite{FP14},
Pinchuk--Shafikov--Sukhov~\cite{PSS},
and many others.
However, the reader should be aware that sometimes in the literature
on singular subvarieties
a Segre variety is defined with respect to a single defining function
and it is not made clear that the Segre variety is then not well-defined if the point
moves.

A good reference for real-analytic geometry is
Guaraldo--Macr\`\i--Tancredi~\cite{TopicsReal}, and a good reference for
complex analytic subvarieties is Whitney~\cite{Whitney:book}.

A real-analytic subvariety of an open $U \subset \C^n$ is a relatively closed subset
$X \subset U$ defined locally by the vanishing of real-analytic functions.
If $p \in X$, then the ideal $I_p(X)$ of real-analytic
germs at $p$ vanishing on $X$ is generated by the components of a mapping $f(z,\bar{z})$.
Let $\Sigma_p X$, the germ of the \emph{Segre variety} at $p$,
be the germ at $p$ of a complex-analytic subvariety given by the vanishing
of $z \mapsto f(z,\bar{p})$ ($\Sigma_p X$ is independent of the generator
$f$).  Normally $\Sigma_p X$ is of the same complex
codimension as is the real codimension of $X$.  So if $X$ is a real
hypersurface,
then $\Sigma_p X$ is usually a germ of a complex hypersurface.
For a hypersurface, we say
$X$ is \emph{Segre-degenerate} at $p$ if $\Sigma_p X$
is not a complex hypersurface, that is, if
$\Sigma_p X = (\C^n,p)$.  See \S~\ref{sec:segre} for a more precise
definition.

One of the main differences of real and complex varieties is that real
varieties need not be coherent.
A real-analytic subvariety is \emph{coherent} if the sheaf of germs of
real-analytic functions vanishing
on $X$ is a coherent sheaf.  Equivalently, $X$ is coherent if it has a
complexification, that is, a single variety that defines
the complexification of all germs of $X$,
or in yet other words, if for every $p$,
representatives of the generators of $I_p(X)$
generate the ideals $I_q(X)$ for all nearby $q$.
For the hypersurface case, we prove the following result.

\begin{thm} \label{thm:mainhyper}
Let $U \subset \C^n$ be open and $X \subset U$ a
real-analytic subvariety of codimension 1 (a hypersurface).
Let $X_{[n]} \subset X$ be the set of Segre-degenerate points.
Then:
\begin{enumerate}[(i)]
\item
$X_{[n]}$ is a semianalytic set of dimension at most $2n-4$, which is locally
contained in a complex-analytic subvariety of (complex) dimension at most $n-2$.
\item
If $X$ is coherent, then 
$X_{[n]}$ is a complex-analytic subvariety of (complex) dimension at most $n-2$.
\end{enumerate}
\end{thm}

The dimension of the complex subvariety may be smaller than $n-2$.
Example~\ref{example:isol} gives a coherent hypersurface in
$\C^3$ where $X_{[n]}$ is an isolated point.
For noncoherent $X$, examples exist
for which $X_{[n]}$ is
not a complex variety, or that are not even a real-analytic subvariety.
In particular, the dimension of $X_{[n]}$ need not be even.
Example~\ref{example:Sbadrealline} is a hypersurface in $\C^3$ such that (real) dimension of
$X_{[n]}$ is 1.  In Example~\ref{example:nonvar}, $X_{[n]}$ is only
semianalytic and not a real-analytic subvariety.

The Segre variety can be defined with respect to a
specific defining function, or a neighborhood $U$ of a point $p$.
For a small enough $U$, take the representatives of the generators of
$I_p(X)$, and use those to define $\Sigma_q^U X$ for all $q \in X$.
The germ of $\Sigma_p^U X$ is then $\Sigma_p X$.
However, for a noncoherent $X$, the
germ of $\Sigma_q^U X$ at $q$ need not be the same as
the germ $\Sigma_q X$,
no matter how small $U$ is and how close $q$ is to $p$,
since the representatives of the generators of $I_p(X)$ may not generate
$I_q(X)$.
There may even be regular points $q$ arbitrarily close
to $p$ where $\Sigma_q^U X$ is singular (reducible) at $q$.
See Example~\ref{example:regularSbad}.
If $q$ is a regular point where $X$ is generic (e.g.\ a hypersurface), the
germ $\Sigma_q X$ is always regular.
The point is that the germs $\Sigma_q X$ cannot be defined coherently by
a single set of equations for a noncoherent subvariety.

The results above are a special case of results for higher codimension.
In general,
the set of ``Segre-degenerate points'' would be points where the Segre
variety is not of the expected dimension.
The main result of this paper is that for
general $X$, we can stratify $X$ into semianalytic sets
by the dimension of the Segre variety.

\begin{thm} \label{thm:main2}
Let $U \subset \C^n$ be open and $X \subset U$ a
real-analytic subvariety of dimension $d < 2n$ (i.e. $X \not= U$).
Let $X_{[k]} \subset X$ be the set of points where the Segre variety
is of dimension $k$ or higher.  Then:
\begin{enumerate}[(i)]
\item For every $k = 0,1,\ldots,n$, $X_{[k]}$ is a closed semianalytic subset of $X$,
and $X_{[n]}$ is locally (as germs at every point) contained in a complex-analytic subvariety
of dimension at most $n-d-1$.
\item If $X$ is coherent, then for every $k=0,1,\ldots,n$, $X_{[k]}$ is a closed real-analytic
subvariety of $X$, and $X_{[n]}$ is a complex-analytic subvariety of
dimension at most $n-d-1$.
\end{enumerate}
\end{thm}

The sets $X_{[k]}$ are nested: $X_{[k+1]} \subset X_{[k]}$.
If $X$ is of pure dimension $d \geq n$, we find that
$X_{[d-n]} = X$.  Then
$X_{[n]} \subset \cdots \subset X_{[k]} \subset \cdots \subset X_{[d-n]}=X$.
If, furthermore, there exists a regular point of $X$ where $X$
is a generic submanifold,
then $X_{[d-n+1]}$
(the reasonable definition of ``Segre-degenerate points'' in this case)
is a semianalytic subset of $X$ of dimension less than $d$,
since where $X$ is a generic submanifold
the dimension of the Segre variety is necessarily $d-n$.
We avoid defining the term Segre-degenerate for general $X$
as the Segre varieties can be degenerate in various ways;
it is better to just talk about the sets $X_{[k]}$
or the sets $X_{[k]} \setminus X_{[k+1]}$.
In any case, since the sets $X_{[k]}$ are semianalytic,
every reasonable definition of ``\emph{Segre-degenerate}'' based on dimension leads
to a semianalytic set.

Notice that for $k < n$, the set $X_{[k]}$ is not necessarily complex even if it is
a proper subset of a coherent $X$, see Example~\ref{example:noncomplex}.

The structure of this paper is as follows.  First, we cover some preliminary
results on subvarieties and semianalytic sets in \S~\ref{sec:prelim}.
We introduce Segre varieties in the singular case in \S~\ref{sec:segre}.
In \S~\ref{sec:coh}, we prove the simpler results for the coherent case,
and we cover the noncoherent case in \S~\ref{sec:noncoh}.
In \S~\ref{sec:examples} we present some of the examples showing that the
results are optimal and particularly illustrating the degeneracy of the noncoherent case.

The author would like to acknowledge Fabrizio Broglia for very insightful
comments and pointing out some missing hypotheses.  The author would also
like to thank the anonymous referee and also Harold Boas for careful reading
of the manuscript and for suggesting quite a few improvements to the exposition.

%%%%%%%%%%%%%%%%%%%%%%%%%%%%%%%%%%%%%%%%%%%%%%%%%%%%%%%%%%%%%%%%%%%%%%%%%%%%

\section{Preliminaries} \label{sec:prelim}

We remark that the content of this section is not new but totally classical,
and the degeneracies shown in the examples have been known for a long time,
already by Cartan, Whitney, Bruhat, and others.  See e.g. \cites{Cartan57,WhitneyBruhat}.

\begin{defn}
Let $U \subset \R^k$ (respectively $U \subset \C^k$) be open.
The set
$X \subset U$ is \emph{a real-analytic subvariety} (resp.\ a
\emph{complex-analytic subvariety}) of $U$ if for each point $p \in U$, there exists a neighborhood
$V \subset U$ of $p$ and a set of real-analytic (resp.\ holomorphic)
functions $\sF(V)$ such that
\begin{equation}
X \cap V = \{ p \in V : f(p) = 0 \text{ for all } f \in \sF(V) \} .
\end{equation}
Write $X_{\text{reg}} \subset X$ for the set of points which are \emph{regular}, that is,
\begin{equation}
\begin{split}
X_{\text{reg}}
\overset{\text{def}}{=}
\{ p \in X : {} &\exists \text{ neighborhood $V$ of $p$, such that }
\\
& V \cap X \text{ is a real-analytic (resp.\ complex) submanifold}  \} .
\end{split}
\end{equation}
The set of \emph{singular points} is the complement:
%\begin{equation}
$X_{\text{sing}}
\overset{\text{def}}{=}
X \setminus X_{\text{reg}}$.
%\end{equation}
The dimension of $X$ at $p \in X_{\text{reg}}$, written as $\dim_p X$,
is the real (resp.\ complex) dimension of the real-analytic (resp.\ complex)
manifold at $p$.  The dimension of $X$, written as $\dim X$,
is the maximum dimension at any regular point.
The dimension of $X$ at $p \in X_{\text{sing}}$ is the minimum dimension of $X \cap V$
over all neighborhoods $V$ of $p$.
Define
\begin{equation}
X^*
\overset{\text{def}}{=}
\{ p \in X_{\text{reg}} : \dim_p X = \dim X \} .
\end{equation}
A variety or germ is \emph{irreducible} if it cannot be written as a union
of two proper subvarieties.
Let $I_q(X)$ denote the ideal of germs $(f,q)$ of functions that vanish
on the germ $(X,q)$.
\end{defn}

An \emph{analytic space} is, like an abstract manifold, a topological space with an
atlas of charts with real-analytic (resp.\ holomorphic) transition maps, but the
local models are subvarieties rather than open sets of $\R^n$ or $\C^n$.
See e.g. \cites{TopicsReal,Whitney:book}.

Subvarieties are closed subsets of $U$.
If a topology on $X$ is required, we take the subspace topology.
%One example of a subvariety is a closed embedded real-analytic submanifold of $U$.
%The set $\sF(V)$ can be taken to be a finite set.
Unlike in the complex case, a real-analytic subvariety
can be a $C^k$-manifold while
being singular as a subvariety.  For example, $x^2-y^{2k+1} = 0$ in $\R^2$.
Also, in the real case, the set of singular points need not be a subvariety and
$X^*$ need not equal $X_{\text{reg}}$.

\begin{defn}[see e.g. \cites{Loja:semi,BM:semisub}]
For a set $V$ (an open set in $\R^n$, or a subvariety),
let $\sS\bigl(C^\omega(V)\bigr)$ be the smallest family of sets (the intersection of all such families) that 
is closed under finite unions, finite intersections, and complements of sets
of the form
\begin{equation}
\bigl\{ x \in V : f(x) \geq 0 \bigr\} ,
\end{equation}
where $f \in C^\omega(V)$ ($f$ real-analytic in $V$, or a restriction
of a real-analytic function if $V$ is a subvariety). 

A set $X \subset U$ is \emph{semianalytic} (in $U$) if for each $p \in U$,
there is a neighborhood $V$ of $p$
such that $X \cap V \in \sS\bigl(C^\omega(V)\bigr)$.
Here $U$ is an open set in $\R^n$, a subvariety, or an analytic
space.
\end{defn}

Note that $\{ x : f(x) \leq 0 \} = \{ x : -f(x) \geq 0 \}$.
Equality is obtained by intersecting 
$\{ x : f(x) \geq 0 \}$ and $\{ x : -f(x) \geq 0 \}$.  Complement 
obtains sets of the form $\{ x : f(x) > 0 \}$ and
$\{ x : f(x) \neq 0 \}$.  Thus we have all
equalities and inequalities.

Subvarieties are
semianalytic, but the family of semianalytic sets is much richer.
If $X$ is a complex-analytic subvariety, then $X_{\text{sing}}$ is a
complex-analytic subvariety,
while if $X$ is only real-analytic, then $X_{\text{sing}}$ is only a
semianalytic subset.

\begin{example}
The Whitney umbrella, $sx^2 = y^2$ in $\R^3$ using coordinates $(x,y,s)$,
%(to avoid using a real $z$),
is a set $X$ where
$X_{\text{sing}}$ is the set given by $x=0$, $y=0$, and $s \geq 0$.
\end{example}

It is a common misconception related to the subject of this paper
to think that the set of singular points of a real subvariety $X$ can be defined
by the vanishing of the derivatives of functions that vanish on $X$.
For a subvariety $X$ defined near $p$, it is
possible that $d\psi$ vanishes on some regular points of $X$ arbitrarily
near $p$ for every function $\psi$ defined near $p$ such that $\psi = 0$ on
$X$.
Before proving this fact, let us prove a simple lemma.

\begin{lemma} \label{lemma:cone}
Suppose $X = \{ x \in \R^k : P(x) = 0 \}$ for an irreducible homogeneous polynomial $P$
(irreducible in the ring of polynomials)
and $X$ is a hypersurface (dimension $k-1$).

If $(f,0)$ is a germ of a real-analytic function
that vanishes on $X$, then $(f,0)$ is a multiple of the germ $(P,0)$.
In other words, $I_0(X)$ is generated by the germ $(P,0)$.
\end{lemma}

\begin{proof}
The proof is standard, it is a version of one of the claims from
the proof of Chow's theorem.
Clearly, $X$ is a real cone, that is, if $x \in X$ then
$\lambda x \in X$ for all $\lambda \in \R$.
Write a representative $f(x) = \sum_{\ell=0}^\infty f_\ell(x)$ in terms of
homogeneous parts.
Suppose $x \in X$, so $f(x) = 0$.  As $\lambda x \in X$,
then $f(\lambda x) \equiv 0$.  But then
$\sum_{\ell=0}^\infty f_\ell(\lambda x) =\sum_{\ell=0}^\infty \lambda^\ell
f_\ell(x)$
is identically zero, meaning $f_\ell(x) = 0$ for all $\ell$.
Since $X$ is a hypersurface, the polynomial $P$ generates the ideal of
all polynomials vanishing
on $X$ and thus $P$ divides all the polynomials $f_\ell$ (See e.g. Theorem 4.5.1
in~\cite{BCR}).
Thence, the germ $(P,0)$ divides the germ $(f,0)$.
\end{proof}

\begin{example} \label{example:regularbad}
Let us give an example of a pure 2-dimensional
real-analytic subvariety $X \subset \R^3$ with an isolated singularity
at the origin, such that
for any real-analytic defining function $\psi$ of $X$ near the origin, the set
where both $d\psi$ and $\psi$ vanish is a 1-dimensional subset of $X$.
Therefore, the set where the derivative vanishes for the defining function
is of larger dimension than the singular set, and $d \psi$ vanishes at some
regular points.  This subvariety will be a useful example later
(Example~\ref{example:regularSbad}), and it is a useful example of a
noncoherent subvariety where coherence breaks not because
of a smaller dimensional component.

Let $X$ be the subvariety of $\R^3$ in the coordinates $(x,y,s) \in \R^3$:
\begin{equation} \label{eq:example1}
(x^2+y^2)^6 - s^8x^3(s-x) = 0 .
\end{equation}
We claim that $X$ is as above.
Despite the singularity being just the origin,
for any real-analytic $\psi$ defined near the origin that vanishes on $X$,
we get $d \psi(0,0,s) = 0$, so the derivative vanishes on
\begin{equation}
\bigl\{ (x,y,s) \in \R^3 : x = 0, y=0 \bigr\} = \{ 0 \} \times \{
0 \} \times \R \subset X.
\end{equation}

As this example will be useful for Segre varieties, we
prove the claim in detail.
The subvariety in $\R^2$ defined by
$(x^2+y^2)^6 - x^3(1-x) = 0$
is irreducible.  Indeed, it is a connected compact submanifold.
To see that it is connected and compact, solve for $y$.
The tricky part is showing that the subvariety is nonsingular near the
origin, which can be seen by writing
$(x^2+y^2)^6 = x^3(1-x)$ and taking the third root to get
\begin{equation}
(x^2+y^2)^2 = x\sqrt[3]{1-x} .
\end{equation}
Near the origin, we can solve for $x$ using the implicit function theorem.

Homogenize $(x^2+y^2)^6 - x^3(1-x)$ with $s$ to get the set $X$ in $\R^3$ given by
\eqref{eq:example1}.
The set $X$ is a cone with an isolated singularity; it is a cone over
a manifold.
By Lemma~\ref{lemma:cone}, if $\psi$ vanishes on $X$, then
\begin{equation}
\psi = 
\bigl( (x^2+y^2)^6 - s^8x^3(s-x) \bigr) \varphi.
\end{equation}
In other words, on $X$, $d \psi$ must vanish where 
the derivative of $(x^2+y^2)^6 - s^8x^3(s-x)$ vanishes.
\end{example}

%%%%%%%%%%%%%%%%%%%%%%%%%%%%%%%%%%%%%%%%%%%%%%%%%%%%%%%%%%%%%%%%%%%%%%%%%%%%

\section{Segre varieties} \label{sec:segre}

Consider subvarieties of $\C^n \cong \R^{2n}$.
Let $U \subset \C^n$ be open and $X \subset U$ a real-analytic subvariety.
Write $U^{\text{conj}} = \{ z : \bar{z} \in U \}$ for the complex
conjugate.  Let $\iota(z) = (z,\bar{z})$ be the embedding of $\C^n$ into
the ``diagonal'' in $\C^n \times \C^n$.  Denote by $\sX^U$ the smallest
complex-analytic subvariety of $U \times U^{\text{conj}}$ such that $\iota(X) \subset \sX^U$.
By smallest we mean the intersection of all such subvarieties.
It is standard that there exists a small enough $U$ (see below)
such that
$\sX^U \cap \iota(\C^n) = \iota(X)$.
Let $\sigma \colon \C^n \times \C^n \to \C^n \times \C^n$ denote the
involution $\sigma(z,w) = (\bar{w},\bar{z})$.  Note that the ``diagonal''
$\iota(\C^n)$ is the fixed set of $\sigma$.

\begin{prop}
Let $U \subset \C^n$ be open and
$X \subset U$ be a real-analytic subvariety.
Then $\sigma(\sX^U) = \sX^U$.
\end{prop}

\begin{proof}
The set $\sigma(\sX^U)$ is a complex-analytic subvariety as it is an defined by
vanishing of anti-holomorphic functions,
and hence by holomorphic functions.  As $X$ is
fixed by $\sigma$, we have $X \subset \sigma(\sX^U) \cap \sX^U$, and the
result follows as $\sX^U$ is the smallest subvariety containing $X$.
\end{proof}

The ideal $I_p(X)$ can be generated by the real and imaginary parts
of the generators of the ideal of germs of holomorphic functions
defined at $(p,\bar{p})$ in the complexification that vanish on the
germ of $\iota(X)$ at $(p,\bar{p})$.
Call the ideal of these holomorphic functions $\sI_p(X)$.

Given a germ of a real-analytic subvariety $(X,p)$, denote by
$\sX_p$ the smallest germ of a complex-analytic subvariety of
$\bigl(\C^n \times \C^n, (p,\bar{p}) \bigr)$ that
contains the image of $(X,p)$ by $\iota$.
The germ $\sX_p$ is called the \emph{complexification} of $(X,p)$.
It is not hard to see that the irreducible components of $(X,p)$
correspond to the irreducible components of $\sX_p$; if $(X,p)$
is irreducible, so is $\sX_p$.
In the theory of real-analytic subvarieties,
$\sX^U$ would not be called a \emph{complexification} of $X$ unless
$\bigl(\sX^U,(p,\bar{p})\bigr) = \sX_p$ for all $p \in X$, and that cannot
always be achieved.

As we will need a specific neighborhood often, we make
the following definition.

\begin{defn}
Let $X \subset U$ be a real-analytic subvariety of dimension $d$ of an open set $U \subset \C^n$.
We say $U$ is \emph{good for $X$ at $p \in X$} if the
following conditions are satisfied:
\begin{enumerate}[(i)]
\item
$U$ is connected.
\item
The real dimension of $(X,p)$ is $d$ and the complex dimension of $\sX_p$
and $\sX^U$ is also $d$.
\item
There exists a real-analytic function $\psi \colon U \to \R^k$
whose complexification converges in $U \times U^{\text{conj}}$,
whose zero set is $X$, and whose germ $(\psi,p)$ generates $I_p(X)$.
\item
$\sX^U \cap \iota(\C^n) = \iota(X)$.
\item
$\bigl( \sX^U , (p,\bar{p}) \bigr) = \sX_p$.
\item
The irreducible components of $\sX^U$ correspond in a one-to-one
fashion to the irreducible components of the germ $\sX_p$.
\end{enumerate}
If $U' \subset U$ is good for $X \cap U'$ at $p$ we say simply that $U'$ is good
for $X$ at $p$.
\end{defn}

\begin{prop}
Suppose $U \subset \C^n$ is open, $X \subset U$ is a real-analytic
subvariety, and $p \in X$.
Then there
exists a neighborhood $U' \subset U$ of $p$ such that $U'$ is good for $X$ at
$p$.

Furthermore, for any neighborhood $W$ of $p$, there exists a neighborhood
$W' \subset W$ of $p$ that is good for $X$ at $p$.
\end{prop}

\begin{proof}
The idea is standard (see e.g. \cite{TopicsReal}), but let us sketch a proof.
The main difficulty is mostly notational.
Take the germ $\sI_p(X)$ of complexified functions that
vanish on the germ of $\iota(X)$ at $(p,\bar{p})$.  Note that
$\sI_p(X)$ is closed under the conjugation taking
$\psi$ to $\overline{\psi \circ \sigma}$, that is,
$\psi(z,\zeta)$ to
$\bar{\psi}(\zeta,z)$.  It is generated by a
finite set of functions $f_1,\ldots,f_k$,
which are all defined in some polydisc $\Delta
\times \Delta^{\text{conj}}$ centered at $p$.  The real and imaginary parts of these
functions also generate an ideal, and this ideal must be equal to $I_p(X)$.
We can also assume that $\Delta$ is small enough that all the components
of the subvariety $\sV$ defined by $f_1,\ldots,f_k$ go through $(p,\bar{p})$
(in other words $\sV$ is the smallest subvariety of $\Delta$ containing
the germ of $\iota(X)$ at $(p,\bar{p})$).  Similarly, make $\Delta$ small
enough that the real and imaginary parts of $f_1,\ldots,f_k$
restricted to the diagonal, give the subvariety $X \cap \Delta$
all of whose components go through $p$.
We can take $\Delta$ to also be small enough that all components of
$\sX_p$ have distinct representatives in $\Delta$.
The set $\Delta$ is our $U'$.
\end{proof}

\begin{defn}
Suppose $U \subset \C^n$ is open and $X \subset U$ is a real-analytic
subvariety.
The \emph{Segre variety} of $X$ at $p\in U$ relative to $U$ is the set
\begin{equation}
\Sigma_p^U X
\overset{\text{def}}{=}
\bigl\{ z \in U : (z,\bar{p}) \in \sX^U \bigr\} .
\end{equation}
If $U' \subset U$, we write $\Sigma_p^{U'} X$ for 
$\Sigma_p^{U'} (X \cap U')$.

When $U'$ is good for $X$ at $p \in X$,
define the germ
\begin{equation}
\Sigma_p X
\overset{\text{def}}{=}
\bigl(\Sigma_p^{U'} X, p \bigr) .
\end{equation}

Define
\begin{align}
X_{[k]}
&
\overset{\text{def}}{=}
\bigl\{ z \in U : \dim \Sigma_z X \geq k \bigr\} ,
\\
X_{U [k]}
&
\overset{\text{def}}{=}
\bigl\{ z \in U : \dim_z \Sigma_z^U X \geq k \bigr\} .
\end{align}
\end{defn}

The germ $\Sigma_p X$ is well-defined by the proposition.  First, there
exists a good neighborhood of $p$, and any smaller good neighborhood of $p$
would give us the same germ of the complexification at $p$.

If $X$ is an irreducible hypersurface,
$X$ is \emph{Segre-degenerate} at $p \in X$ if $\Sigma_p X =
(\C^n,p)$, that is, if $p \in X_{[n]}$.
A point $p$ is \emph{Segre-degenerate relative to $U$} if
$\dim_p \Sigma_p^U X = n$,
that is, if $p \in X_{U [n]}$.  A key point of this paper is that
these two notions can be different.  We will see that $X_{U[n]}$ is always a
complex subvariety and contains $X_{[n]}$, and the two are not necessarily
equal even for a small enough $U$.
They may not be even of the same dimension.

For a general dimension $d$ set, we will simply talk about the sets $X_{[k]}$
and we will not make a judgement on what is the best definition for the word
``Segre-degenerate.''

A (smooth) submanifold is called \emph{generic} (see \cite{BER:book})
at $p$
%at a point if given the defining functions
%$\rho_1,\ldots,\rho_k$,
%the holomorphic parts of the exterior derivative,
%$\partial \rho_1, \ldots, \partial \rho_k$,
%are linearly independent.  Equivalently, it is generic
if in some local holomorphic coordinates $(z,w) \in \C^{n-k} \times \C^k$ vanishing at $p$ it is
defined by
\begin{equation}
\Im w_1 = r_1(z,\bar{z},\Re w), \quad \ldots, \quad
\Im w_k = r_k(z,\bar{z},\Re w),
\end{equation}
with $r_j$ and its derivative vanishing at $0$.
For instance, a hypersurface is generic.

\begin{prop} \label{prop:genericman}
Suppose $X$ is a real-analytic submanifold of $\C^n$ of dimension $d$ (so codimension
$2n-d$) and $p \in X$.
\begin{enumerate}[(i)]
\item $\dim \Sigma_p X \leq \frac{d}{2}$.
In particular, if $d < 2n$, then $X_{[n]} = \emptyset$.
\item If $X$ is generic at $p$,
then $\Sigma_p X$ is a germ of a complex submanifold and $\dim \Sigma_p X = d-n$.
\end{enumerate}
\end{prop}

\begin{proof}
We start with the generic case.
Using the defining functions above, $k=2n-d$, we note that if we plug in $\bar{w}=0$
and $\bar{z}=0$, we get $k$ linearly independent defining equations for a complex 
submanifold.

If $X$ is not generic, then we can write down similar equations but solve
for real and imaginary parts some of the variables.
Since these could conceivably be the real and
imaginary parts of the same variable, we may only get $\frac{k}{2} = \frac{2n-d}{2}$
independent equations, so the dimension of $\Sigma_p X$ could be as high as
$n - \frac{2n-d}{2} = \frac{d}{2}$.
\end{proof}

If $X$ is regular at $p$ but not generic,
the germ $\Sigma_p X$ could possibly be singular
and the dimension may vary as $p$ moves on the submanifold.  See
Example~\ref{example:mfld}.

Let us collect some basic properties of Segre varieties in the singular case.

\begin{prop} \label{prop:basicprops}
Let $X \subset U \subset \C^n$ be a real-analytic subvariety of dimension
$d$ and $p \in X$.  Then
\begin{enumerate}[(i)]
\item
$\displaystyle \Sigma_p X \subset \bigl( \Sigma_p^U X , p \bigr)$.
\item
$\displaystyle \dim \Sigma_p X \geq d-n$.
\item
$\displaystyle X_{[k]} \subset X_{U [k]}$
for all $k$.
\item
If $U$ is good for $X$ at $p$ and $d \leq n$, then $\dim \Sigma_p^U X < n$.
\item
If $d \leq n$, then $X_{[n]} = \emptyset$.
\item
$q \in \Sigma_q X$ if and only if $q \in X$, and so
$X_{[k]} \subset X$ for all $k = 0,1,2,\ldots,n$.
\item
If $U$ is good for $X$ at $p$, then
$q \in \Sigma_q^U X$ if and only if $q \in X$, and so
$X_{U [k]} \subset X$ for all $k = 0,1,2,\ldots,n$.
\end{enumerate}
\end{prop}

\begin{proof}
If $U' \subset U$ then $\Sigma_p^{U'} \subset
\Sigma_p^U$, as any analytic function defined on $U$ that vanishes on $X$ is
an analytic function on $U'$ that vanishes on $X \cap U'$.  Parts (i) and
(iii) follow.

For part (ii), the complexification $\sX_p$ has dimension $d$.  Let $U'$ be
good for $X$ at $p$.  The germ
$\Sigma_p X$ is the germ at $p$ of the intersection of $\sX^{U'}$ and
$\C^n \times \{ \bar{p} \}$.  The codimension of $\sX^{U'}$ at $(p,\bar{p})$
in $\C^n \times \C^n$ is $2n-d$, and the codimension of 
$\C^n \times \{ \bar{p} \}$ is $n$.  Hence, their intersection is of
codimension at most $3n-d$ or dimension $2n-(3n-d) = d-n$.

To see (iv) first note that if $d < n$, 
then it is impossible for $\Sigma_q^U
X$ to be of dimension $n$ as it is a subvariety of $\sX^U$, which
is of dimension $d < n$.  If $d=n$, then without loss of generality
suppose that $(X,p)$ is irreducible.
As $U$ is good for $X$ at $p$, then $\sX^U$ is also irreducible.
By dimension, as $\sX^U$ is of dimension $n$ and $\dim \Sigma_p X = n$,
then $U \times \{ \bar{p} \}$ would be an irreducible component of $\sX^U$.
By symmetry (applying $\sigma$),
$\{ p \} \times U^{\text{conj}}$ is also an irreducible component of $\sX^U$.
This is a contradiction as
$\sX^U$ is irreducible.

Then (v) follows from (iv) by considering a small enough good neighborhood
of every $q \in X$.

For (vi),
if $q \in X$, then $q \in \Sigma_q X$, since $\psi(q,\bar{q}) = 0$ for any
germ of a function at $q$ that vanishes on $X$.  If $q \notin X$, then
clearly $\Sigma_q X = \emptyset$.  So
$X_{[k]} \subset X$.

For (vii), again if $q \in X$, then it must be that $q \in \Sigma_q^U X$.
Similarly, for a good $U$, we have $\sX^U = \iota(X)$, and
so $q \in \Sigma_q^U X$ means that $q \in X$.
\end{proof}

The point of this paper is that even for arbitrarily small
neighborhoods $U$ of $p$ (even good for $X$ at $p$) and a $q \in U$ that is arbitrarily close
to $p$, it is possible that
\begin{equation}
(\Sigma_p^U X,q) \not= \Sigma_q X.
\end{equation}
That is, unless $X$ is coherent.
Let us focus on $X_{[n]}$ for a moment.
It is possible that for all neighborhoods $U$ of a
point $p \in X$,
\begin{equation}
X_{[n]} \not= X_{U [n]} .
\end{equation}
The set $X_{U [n]}$ is rather well-behaved.

\begin{prop} \label{prop:XUnsubvariety}
Let $X \subset U \subset \C^n$ be a real-analytic subvariety of dimension
$d < 2n$,
and suppose $U$ is good for $X$ at some $p \in X$.
Then 
$X_{U [n]}$ is a complex-analytic subvariety of dimension at most $n-d-1$.
In particular, $X_{[n]}$ is contained in a complex-analytic subvariety of
dimension at most $n-d-1$.
\end{prop}

\begin{proof}
Without loss of generality, suppose that $(X,p)$ is irreducible.
The variety $\sX^U$ is fixed by the involution $\sigma$.
In other words,
$(z,\bar{w}) \in \sX^U$ if and only if
$(w,\bar{z}) \in \sX^U$.
So if $q \in X_{U[n]}$, then
$(z,\bar{q}) \in \sX^U$ for all $z \in U$,
and therefore
$(q,\bar{z}) \in \sX^U$ for all $z \in U$.
In particular, $q \in \Sigma^U_z X$ for all $z \in U$.
As $\Sigma^U_z X$ is a complex-analytic subvariety, generically of dimension $d-n$,
then $X_{U[n]}$ is a complex-analytic subvariety of dimension at most $d-n$.

The only way that $X_{U[n]}$ could be of dimension $d-n$ is if 
all the varieties $\Sigma^U_z X$ contained a fixed complex-analytic
subvariety $V$ of dimension $d-n$.
This means that
$V \times \C^n \subset \sX^U$ and
$\C^n \times V^{\text{conj}} \subset \sX^U$ (by applying $\sigma$).
By dimension, these are components of $\sX^U$.
Since we assumed that $(X,p)$ is irreducible, so is $\sX_p$
and so is $\sX^U$ if $U$ is good for $X$ at $p$,
and we obtain a contradiction.
Hence, $X_{U[n]}$ must be of dimension at most $d-n-1$.
\end{proof}

%%%%%%%%%%%%%%%%%%%%%%%%%%%%%%%%%%%%%%%%%%%%%%%%%%%%%%%%%%%%%%%%%%%%%%%%%%%%

\section{Coherent varieties} \label{sec:coh}

A real-analytic subvariety is \emph{coherent} if the sheaf of germs of real-analytic
functions vanishing on $X$ is a coherent sheaf.
The fundamental fact about coherent subvarieties is that they possess
a global complexification.  That is, if $X$ is coherent,
then there exists a complex-analytic subvariety $\sX$ of some neighborhood of $X$ in
$\C^n \times \C^n$ such that $\sX \cap \iota(\C^n) = X$ and
$\bigl(\sX,(p,\bar{p})\bigr)$ is equal to $\sX_p$, the complexification of the germ
$(X,p)$ at every $p \in X$.  See \cite{TopicsReal}.

\begin{lemma}
Let $X \subset U \subset \C^n$ be a real-analytic subvariety.
If $X \subset U$ is coherent and $U$ is good for $X$ at $p \in X$,
then $\Sigma_q X = (\Sigma^U_q X,q)$ for all $q \in X$.
In particular, $X_{U[n]} = X_{[n]}$.
\end{lemma}

\begin{proof}
Since $X$ is coherent, we have a global complexification $\sX$
and hence $\sX^U = \sX \cap U$.  In particular, this is true for any 
good neighborhood $U' \subset U$ of any point $q \in X$, so
$\Sigma_q X = (\Sigma^{U'}_q X,q) = (\Sigma^U_q X,q)$.

As $X_{U[n]} \subset X$ and it is the set where
$(\Sigma^U_q X ,q)$ is of dimension $n$, we find that it is equal to the
set where $\Sigma_q X$ is of dimension $n$.  In other words,
$X_{U[n]} = X_{[n]}$.
\end{proof}

We can now prove the theorem for coherent subvarieties.
The following theorem implies the coherent part of Theorem~\ref{thm:main2} for $k=n$,
and hence the coherent part of Theorem~\ref{thm:mainhyper}.

\begin{thm}
Let $U \subset \C^n$ be open and $X \subset U$ be a coherent real-analytic
subvariety of dimension $d < 2n$.  Then $X_{[n]}$ is a complex-analytic
subvariety of dimension at most $d-n-1$.
\end{thm}

\begin{proof}
It is sufficient to work in a good neighborhood of some point, without loss
of generality, assume that $U$ is good for some $p \in X$.  Apply the
lemma and Proposition~\ref{prop:XUnsubvariety}.
\end{proof}

For general $k$, we have the following theorem, which finishes the coherent case
of Theorem~\ref{thm:main2} for $k < n$.
That is,
for every $k$, the $X_{[k]}$ sets are subvarieties of $X$ for coherent $X$.
These subvarieties no longer need to be complex-analytic.

\begin{thm} \label{thm:cohXk}
Let $U \subset \C^n$ be open and $X \subset U$ be a coherent real-analytic
subvariety.  Then for every
$k=0,1,\ldots,n$, $X_{[k]}$ is a real-analytic subvariety of $X$.
\end{thm}

A generic submanifold has the Segre variety of the least possible
dimension.  Let $X$ be an irreducible coherent subvariety of dimension $d$.
If $X_{\text{reg}}$ is generic at some point,
then $\Sigma_q X$ is of (the minimum possible) dimension $d-n$ somewhere.
The Segre-degenerate set is the set where $\Sigma_q X$ is higher than $d-n$,
that is, it is the set $X_{[d-n+1]}$.
According to this theorem, this Segre-degenerate set $X_{[d-n+1]}$ is
a real-analytic subvariety of $X$.

\begin{proof}
It is a local result and so without loss of generality,
assume that $U$ is good for $X$ at some $p \in X$.
%Find $\sX^U$, which is the complexification of $X$ in $U$
%($X$ is coherent).
Let $(z,\xi)$ be the complexified variables of $\C^n \times \C^n$.
Consider the projection $\pi(z,\xi)=\xi$ defined on $\sX^U$.
The Segre variety $\Sigma_z^U X$ is (identified with) the fiber
$\pi^{-1}\bigl(\pi(z,\bar{z})\bigr)$.
The dimension of the germ $\Sigma_z X$ is the dimension at $(z,\bar{z})$
of $\pi^{-1}\bigl(\pi(z,\bar{z})\bigr)$ as $X$ is coherent.
For any integer $k$, the set
\begin{equation}
V_k = \Bigl\{ (z,\xi) \in \sX^U : \dim_{(z,\xi)}
\pi^{-1}\bigl(\pi(z,\xi)\bigr) \geq k \Bigr\}
\end{equation}
is a complex-analytic subvariety of $\sX^U$ (see e.g. Theorem 9F in chapter 7 of
Whitney~\cite{Whitney:book}).
Then $V_k \cap \{ \xi =
\bar{z} \} = \iota(X_{[k]})$ is a real-analytic subvariety of $\iota(X)$.
\end{proof}

%%%%%%%%%%%%%%%%%%%%%%%%%%%%%%%%%%%%%%%%%%%%%%%%%%%%%%%%%%%%%%%%%%%%%%%%%%%%

\section{The set of Segre-degenerate points is semianalytic} \label{sec:noncoh}

It is rather simple to prove that $X_{[k]}$ is always closed (in classical,
not Zariski, topology).

\begin{prop} \label{prop:Xkclosed}
Let $U \subset \C^n$ be open and $X \subset U$ be a real-analytic
subvariety.
Then $q \mapsto \dim \Sigma_q X$ is
an upper semi-continuous function on $X$.
In particular,
for every $k$, $X_{[k]}$ is closed.
\end{prop}

\begin{proof}
Let $p \in X$ be some point and let $U$ be good for $X$ at $p$
and follow the construction in the proof of
Theorem~\ref{thm:cohXk}, that is let $\sX^U$
and $\pi$ be as before.
The Segre variety $\Sigma_z^U X$ is the fiber
$\pi^{-1}\bigl(\pi(z,\bar{z})\bigr)$.  For all $z \in X$,
$\Sigma_z X$ is a subset (possibly proper as $X$ is not coherent) of the germ
$(\Sigma_z^U X,z)$, and so
$\dim_z \Sigma_z^U X \geq \dim \Sigma_z X$.
As $U$ is good for $X$ at $p$,
$(\Sigma_p^U X,p) = \Sigma_p X$.
As the sets $V_k \cap \iota(X)$ are closed,
$\dim \Sigma_p X = (\Sigma_p^U X,p)$ is bounded below by
dimensions of $(\Sigma_q^U X,q)$ for all sufficiently nearby $q$, and
these are in turn bounded below by $\dim \Sigma_q X$.
\end{proof}

We need some results about semianalytic subsets.  We are going to use
normalization on $\sX^U$ and so we need to prove that semianalytic sets are preserved
under finite holomorphic mappings. The key point in that proof is the
following theorem on projection of semialgebraic sets extended to handle
certain semianalytic sets.

\begin{thm}[\L{}ojasiewicz--Tarski--Seidenberg (see
\cites{Loja:semi,BM:semisub})]
Let $\sA$ be a ring of real-valued functions on a set $U$, and let $\pi \colon
U \times \R^m \to U$ be the projection.

If $X \in \sS\bigl(\sA[t_1,\ldots,t_m]\bigr)$, then $\pi(X) \in \sS(\sA)$.
\end{thm}

Complex-analytic subvarieties are preserved under finite (or just proper)
holomorphic maps.  Real semialgebraic sets are preserved under all
real polynomial maps.  On the other hand real-analytic subvarieties or
semianalytic sets are not preserved by finite or proper real-analytic maps.
But, as long as the map is holomorphic and
finite, semianalytic sets are preserved.  Here is an intuitive
useful argument of why this is expected:
Map forward the complexification of a real-analytic subvariety
by the complexification of the map
$(z,\bar{z}) \mapsto \bigl(f(z),\bar{f}(\bar{z})\bigr)$, which is still
finite, so it maps the complexification to a complex-analytic subvariety.
So the image of a real-analytic subvariety of dimension
$d$ via a finite holomorphic map is contained in a real-analytic subvariety
of dimension $d$.  To get equality we need to go to semianalytic sets: Think
of $z \mapsto z^2$ as the map and the real line as the real-analytic subvariety.
The holomorphicity is required as the complexification of a finite
real-analytic map need not be finite
(simple example: $z \mapsto z\bar{z} + i(z+\bar{z})$).

\begin{lemma}
Let $V,W$ be complex analytic spaces, $S \subset V$ a semianalytic set,
and $f \colon V \to W$ a finite holomorphic map.  Then $f(S)$ is
semianalytic of the same dimension as $S$.
\end{lemma}

\begin{proof}
Without loss of generality, assume that $f(V) = W$.
Furthermore, since the map is finite, and finite unions of semianalytic sets
are semianalytic, assume that $V,W$ are actual complex-analytic subvarieties by
working locally in some chart, and in general we can just assume we are
working in an arbitrarily small neighborhood of the origin $0 \in V$,
and that $f(0) = 0$.
Suppose $V$ is a subvariety of some neighborhood $U \subset \C^n$,
and $W$ is a subvariety of some open set $U' \subset \C^m$.
By adding components to $f$ equal to the defining functions of $V$ (and thus
possibly increasing $m$) we can assume without loss of generality that $f
\colon U \to \C^m$ is a finite map on $U$ and not just $V$.

Consider the graph $\Gamma_f$ of $f$ in $U \times \C^m$.  As
$f$ is finite, the projection of $\Gamma_f$ to $\C^m$ is finite.
Hence, the variety $\Gamma_f$ can be defined by functions that are
polynomials in the first $n$ variables (in fact polynomials in the first $n$
variables and a few of the last $m$ variables depending on the codimension of
$W = f(V)$ in $\C^m$).  Let $z = x+iy$ denote the first $n$ variables, and $\xi$
denote the last $m$ variables.  The variety $\Gamma_f$ as a real subvariety
is defined by functions that are polynomials in $x$ and $y$.

Also assume that $U$ is small enough so that $S$ is defined by
real-analytic functions in $U$, that is, $S \in \sS\bigl(C^\omega(U)\bigr)$.
The set $S$ corresponds to a semianalytic set $\widetilde{S} \subset
\Gamma_f$.  The set $\widetilde{S}$ is defined by functions defined in
some $U \times U'$, suppose $\varphi$ is one of these functions.
The subvariety $\Gamma_f$ is defined by polynomials in $x$ and $y$,
so we find Weierstrass polynomials in every one of
$x$ and $y$ with coefficients real-analytic functions in $\xi$
that are in the real-analytic ideal for $\Gamma_f$ at $(0,0)$.
Since adding anything in the ideal does not
change $\varphi$ where it matters (on $\Gamma_f$), we can divide by
these polynomials and find a remainder $\psi$, which is a polynomial in
$x$ and $y$ such that $\psi = \varphi$ on $\Gamma_f$.  In other words,
$\widetilde{S} \in \sS\bigl(C^\omega(U')[x,y]\bigr)$.
By the \L{}ojasiewicz--Tarski--Seidenberg theorem, the projection of
$\widetilde{S}$ to $U'$ is semianalytic.

The fact that the dimension is preserved follows from $f$ being finite.
\end{proof}

\begin{remark}
The conclusion of the lemma is not true if $f$ is not holomorphic and finite.  If $f$ is
proper but not holomorphic, the best we can conclude is that $f(S)$ is 
subanalytic as long as we also assume that $S$ is precompact.  Our task
would be easier if we only desired to prove that $X_{[k]}$ is
subanalytic.
\end{remark}

The proof that $X_{[k]}$ is semianalytic for non-coherent subvarieties is
similar to Theorem~\ref{thm:cohXk}, but we work on the normalization of
the complex variety $\sX^U$.

\begin{thm} \label{thm:noncohXk}
Let $U \subset \C^n$ be open and $X \subset U$ be a real-analytic
subvariety.  Then for every
$k=0,1,\ldots,n$, $X_{[k]}$ is a closed semianalytic subset of $X$.
\end{thm}

\begin{proof}
Again, it is a local result, so without loss of generality,
assume that $U$ is good for $X$ at some
$p \in X$ and suppose that $X$ is irreducible at $p$ and that $X$ is
of dimension $d$.
Consider $h \colon \sY \to \sX^U$, the normalization of
$\sX^U$.
There are two reasons why $\sX^U$ is not the complexification at some point $q$.
For points $z$ arbitrarily near $q$, either the set $X$ is of lower dimension at $z$,
or there are multiple irreducible components of the germ $\bigl(\sX^U,(z,\bar{z})\bigr)$.

Let $\overline{X^*}$ denote the relative closure in $U$ of the set of points
of dimension $d$.  The set $X \setminus \overline{X^*}$ is semianalytic,
and so locally near any $q \in X$ it is possible to write $X = \overline{X^*} \cup X'$ for
$X'$ a real-analytic subvariety of lower dimension (possibly empty) defined in a neighborhood
of $q$.  Suppose for induction
that $X'_{[k]}$ is semianalytic.
Then 
$X'_{[k]} \setminus \overline{X^*} =
X_{[k]} \setminus \overline{X^*}$ is also semianalytic (in a neighborhood of $q$).
In other words,
it remains to prove that $X_{[k]} \cap \overline{X^*}$ is semianalytic.

Let $X_1 = h^{-1}\bigl(\iota(\overline{X^*})\bigr)$, and note that
this is a closed semianalytic subset of $\sY$ of dimension $d$, although it
can have points of various dimensions.  Therefore, take $X_2 =
\overline{X_1^*}$ to be the closure (in $\sY$) of the nonsingular points of
$X_1$ of dimension $d$.  It is clear that $h(X_2) = \iota(\overline{X^*})$.

Let $(z,\xi)$ be the complexified variables of $\C^n \times \C^n$, where
$\sX^U$ lives.
Consider the projection $\pi(z,\xi)=\xi$ defined on $\sX^U$.
The Segre variety $\Sigma_z^U X$ is the fiber
$\pi^{-1}\bigl(\pi(z,\bar{z})\bigr)$, but the germ at $(z,\bar{z})$ may
contain other components, so we pull back to $\sY$.

Let $\eta$ be the variable on $\sY$ and we pull back via $h$ as
$(\pi \circ h)^{-1}\bigl(\pi \circ h(\eta)\bigr)$.
The space $\sY$ is normal and so the germ $(\sY,\eta)$ is irreducible for
all $\eta$.  Near some $\eta \in X_2$, the set $X_2$ is a totally-real
subset of $\sY$ of dimension $d$.
Hence $(\sY,\eta)$, which is irreducible and of dimension $d$, contains $(X_2,\eta)$
and is then the smallest complex subvariety containing
$(X_2,\eta)$.
The germ of the complexification of $X$ at $h(\eta)$ has as its components the
images of $(\sY,\eta')$ via $h$ for all
$\eta' \in h^{-1}(h(\eta)) \cap X_2$.  These images must be contained in the
complexification and as $h(X_2) = \iota(\overline{X^*})$, their 
union is the entire complexification of $X$ at $h(\eta)$.
We thus need to consider the sets
\begin{equation}
W_k = \Bigl\{ \eta \in \sY : \dim_{\eta}
(\pi \circ h)^{-1}\bigl(\pi \circ h(\eta)\bigr)
\geq k \Bigr\} ,
\end{equation}
which are again complex analytic.  We are interested in the sets
$X_2 \cap W_k$, which are semianalytic, and we have just proved above that
$h(X_2 \cap W_k) = \iota(X_{[k]})$.  As $h$ is finite and 
$X_2 \cap W_k$ is semianalytic, we are finished.
\end{proof}

%%%%%%%%%%%%%%%%%%%%%%%%%%%%%%%%%%%%%%%%%%%%%%%%%%%%%%%%%%%%%%%%%%%%%%%%%%%%

\section{Examples of Segre variety degeneracies} \label{sec:examples}

\begin{example} \label{example:isol}
The set of Segre-degenerate points of a coherent hypersurface
in $\C^n$ can be a complex subvariety of dimension strictly less than $n-2$.
Let $X \subset \C^3$ in coordinates $(z,w,\xi) \in \C^3$ be given by
\begin{equation}
z\bar{z} + w\bar{w} - \xi \bar{\xi} = 0.
\end{equation}
The set of regular points is everything except the origin, so
only the origin can be Segre-degenerate, and for this subvariety, it is,
as the above equation generates the ideal $I_0(X)$ by
Lemma~\ref{lemma:cone}.  So $X_{[3]} = \{ 0 \}$, which is of dimension $n-3
= 0$.
\end{example}

\begin{example} \label{example:noncomplex}
For a higher codimensional subvariety, the set $X_{[k]}$ for $k < n$ is
generally not complex.  Clearly if $k \leq d-n$, then $X_{[k]} = X$
and $X$ is not necessarily complex.  But even for higher $k$ less than $n$,
the set need not be complex.
Let $X \subset \C^3$ in coordinates $(z,w,\xi) \in \C^3$ be given by
\begin{equation}
z\bar{z} - w\bar{w} = 0, \qquad \Im \xi = 0 .
\end{equation}
The subvariety $X$ is $4$-dimensional and coherent.
It is easy to see that $X_{[1]} = X$, $X_{[2]} = \{ z=0, w=0, \Im \xi = 0 \}$,
and $X_{[3]} = \emptyset$.  The set $X_{[2]}$ is not complex.
\end{example}

\begin{example} \label{example:mfld}
A submanifold may be Segre-degenerate, if it is CR singular.  Let $(z,w)$ be
the coordinates in $\C^2$ and consider the manifold $X$ given by
\begin{equation}
w = z\bar{z} .
\end{equation}
As this is a complex equation, to find the generators of the ideal, we must
take the real and imaginary parts, or equivalently, also consider the conjugate of
the equation, $\bar{w} = z\bar{z}$.  For points where $z \not= 0$, the Segre
variety is just the trivial point, so zero dimensional.  But at the point
$(0,0)$ the Segre variety is the complex line $\{ w=0 \}$.  In other words,
$X_{[0]} = X$, $X_{[1]} = \{ (0,0) \}$, and $X_{[2]} =
\emptyset$.

Similarly, the Segre variety of a submanifold can be singular if the
manifold is CR singular.  Let $(z,w,\xi)$ be coordinates in $\C^3$ and consider
$X$ given by
\begin{equation}
w = z^2+\bar{z}^2+\xi^2+\bar{\xi}^2 .
\end{equation}
The Segre subvariety at the origin $\Sigma_0 X$ is the pair of complex lines given by
$\{ w = 0, (z+i\xi)(z-i\xi) = 0 \}$.
\end{example}

\begin{example} \label{example:regularSbad}
Consider Example~\ref{example:regularbad}, that is
$(x^2+y^2)^6-s^8x^3(s-x) = 0$ and extend it to $\C^2$ using
$z = x+iy$ and $w = s+it$.
In other words, we use $X \times \R$ if $X$ is the variety of the previous example.
That is, let $X$ in $(z,w) \in \C^2$ be given by
\begin{equation}
f(z,w,\bar{z},\bar{w}) = (z\bar{z})^6-(\Re w)^8(\Re z)^3(\Re w-\Re z) = 0 .
\end{equation}
Similarly to Example~\ref{example:regularbad},
this $f$ generates the ideal at $I_0(X)$,
its derivatives vanish when $z=0$, but $X$ is regular outside of $\{ z = 0 ,
\Re w = 0 \}$.  So there are regular (hypersurface, thus generic) points of $X$
where the complexified $f$
defines a singular subvariety.  That is, regular points of $X$ where the
corresponding $\sX^U$ is singular for any neighborhood $U$ of $0$.
For such a point $q$, for any $U$, $\Sigma_q X$ is regular, but $\Sigma_q^U
X$ is singular at $q$.  In particular,
\begin{equation}
\Sigma_q X \subsetneq (\Sigma_q^U,q) .
\end{equation}
So $\Sigma_q X$ is just one component of the germ
$(\Sigma_q^U,q)$.
\end{example}

\begin{example} \label{example:nonvar}
The set of Segre-degenerate points of a hypersurface need not be
a subvariety for noncoherent $X$.
Let $X \subset \C^3$ in coordinates $(z,w,\xi) \in \C^3$ be given by
\begin{equation}
z\bar{z} - (\xi+\bar{\xi}) w \bar{w} = 0 .
\end{equation}
The set is reminiscent of the Whitney umbrella.  The set $U = \C^3$ is a
good neighborhood for $X$ at $0$.
The set of Segre-degenerate points with respect to $U$ (actually any
neighborhood $U$ of the origin), is $X_{U[3]} = \{ w = z = 0 \}$,
that is,
a one-dimensional complex line.
However, where $\Re \xi < 0$, the variety $X$ is
locally just the line $\{w=z=0 \}$.  Therefore, the variety is a real
manifold of dimension 2 (complex manifold of dimension 1).  At such points
$\Sigma_p X$ is one-dimensional and such points are not in
$X_{[3]}$ (not Segre-degenerate).  Hence,
\begin{equation}
X_{[3]} = \{ (z,w,\xi) \in X : w=z=0, \Re \xi \geq 0 \}
\end{equation}
and this set is not a subvariety, it is a semianalytic set.
\end{example}

\begin{example} \label{example:Sbadrealline}
Let us construct the promised noncoherent
hypersurface in $\C^3$ where the set $X_{[n]}$
of Segre-degenerate points is not complex, in fact, it is a real line.

Let $X \subset \C^3$ in coordinates $(z,w,\xi) \in \C^3$ be given by
\begin{equation}
\psi = 
w^2\bar{w}^2 (\Re \xi)
+
4(\Re z)(\Re \xi)^2 w \bar{w}
+
4(\Re z)^3 z \bar{z}
=
0 .
\end{equation}
The function is irreducible as a polynomial and homogeneous and thus
$(X,0)$ is irreducible as a germ of a real-analytic subvariety.

The set where $d\psi = 0$ is given by $\Re z = 0$, $w=0$, and this set lies
in $X$.  Therefore, $\{ d \psi = 0 \} \subset X$ is 3-real dimensional.  However,
the singular set $X_{\text{sing}}$ is 2-dimensional given by
$\Re z = 0$, $w=0$, and $\Re \xi = 0$. 
Let us prove this fact.
For simplicity let $z=x+iy$ and $\xi = s+it$ and assume $s \not=0$.
Solve for $w\bar{w}$ as
\begin{equation}
w\bar{w} =
%\frac{-4xs^2\pm\sqrt{16x^2s^4-16 s x^3 (x^2+y^2)}}{2s}
%=
x\left(-2s\pm
\frac{2}{s}\sqrt{s^4-s x (x^2+y^2)} \right).
%\frac{-(\Re z)(\Re \xi)^2\pm\sqrt{(\Re z)^2(\Re \xi)^4-4 (\Re \xi) (\Re z)^3 z \bar{z}}}{2(\Re \xi)}
\end{equation}
When the sign is negative and $s\not=0$, we can solve for $x$ by the implicit function
theorem and the subvariety has a regular point there.
If the sign is positive and $s\not=0$, then we claim that there is no solution except
$x=0, s=0, w=0$.  We must check a few possibilities.
If $x > 0$ and $s > 0$, then
$\frac{2}{s}\sqrt{s^4-s x (x^2+y^2)} < 2s$, and as $w \bar{w}$ must be
positive there are no such real solutions.
Similarly for every other sign combination.  That means that the only
solution when $s \not= 0$ is when the sign is positive.  So $X$ is regular when
$\Re \xi = s \not= 0$.
Similarly, it is not difficult to show that $X$ is singular 
at points where $\Re z = 0$, $w=0$, $\Re \xi = 0$:  For example, at such
points, were they regular, the $\Re z =0$ hyperplane
and the $\Re z = - \Re \xi$ hyperplane would both have to be
tangent as their intersections with $X$ are singular (both reducible).
That is impossible for a regular point.

Since $\psi$ generates the ideal at the origin, it is
easy to see that
$X_{U[n]} = \{ z = 0, w=0 \}$ near the origin
for any good neighborhood $U$ of the origin.
As $X_{[n]} \subset X_{U[n]}$ and $X_{[n]} \subset X_{\text{sing}}$,
we can see that $X_{[n]} \subset
\{ z=0, w=0, \Re \xi = 0 \}$.  Since the defining
function does not depend on $\Im \xi$, all the
points of the set $\{ z=0, w=0, \Re \xi = 0 \}$ are in $X_{[n]}$ or none of
them are.  The origin is definitely Segre-degenerate as $\psi$
is the generator of the ideal there, and thus
$X_{[n]} = \{ z=0, w=0, \Re \xi = 0 \}$.
So the set $X_{[n]}$ where $X$ is
Segre-degenerate is of real dimension 1.

In other words:
\begin{enumerate}[(i)]
\item $\dim X_{\text{sing}} = 2$.
\item $\{ d f = 0 \} \cap X$ is 3 real-dimensional for every real-analytic
germ $f$ vanishing on $X$ (and not identically zero).
\item The set of Segre-degenerate points $X_{[n]}$ is a real one-dimensional line.
\item The set of Segre-degenerate points relative to $U$, $X_{U[n]}$, is a complex one-dimensional line at the origin
for every good neighborhood $U$ of the origin, and $X_{U[n]} \cap
X_{\text{reg}}
\not=\emptyset$.
\end{enumerate}
\end{example}

%\section{Declarations}
%
%\noindent
%\textbf{Conflicts of interest/Competing interests:} There exists no conflict
%of interest.
%
%\noindent
%\textbf{Availability of data and material:} Not applicable.

%%%%%%%%%%%%%%%%%%%%%%%%%%%%%%%%%%%%%%%%%%%%%%%%%%%%%%%%%%%%%%%%%%%%%%%%%%%%

\def\MR#1{\relax\ifhmode\unskip\spacefactor3000 \space\fi%
  \href{http://www.ams.org/mathscinet-getitem?mr=#1}{MR#1}}

%%%%%%%%%%%%%%%%%%%%%%%%%%%%%%%%%%%%%%%%%%%%%%%%%%%%%%%%%%%%%%%%%%%%%%%%%%%%

\end{document}